\documentclass[12pt]{amsart}
\usepackage{color,graphicx,wrapfig}
\usepackage{eucal}

\usepackage{color}
\usepackage{graphics}
\usepackage{amsmath}
\usepackage{amsfonts}
\usepackage{graphicx}
\usepackage{mathrsfs}

\newtheorem{theo}{Theorem}[section]
\newtheorem{prop}[theo]{Proposition}
\newtheorem{lem}[theo]{Lemma}
\newtheorem{dis}[theo]{Discussion}

\newtheorem{rem}[theo]{Remark}
\newtheorem{assu} {Assumption}[section]

\newcommand{\N}{\mathbb{N}}

\title[Free Boundaries for Convertible Bonds]{Analysis of  Free Boundaries for Convertible Bonds, with a call feature}
\author{ Sadna Sajadini}
\address{ Department of Mathematics, Royal Institute of Technology, Stockholm}

\date{\today}
\email{sadna@math.kth.se}
\thanks{The author would like to deeply thank Professor Henrik Shahgholian for  suggesting the problem and teaching me patiently, all the details in this work. She also thanks  Dr. John Andersson for his fruitful comments.}


\begin{document}
\date{\today}
\maketitle

\begin{abstract} Convertible bonds give rise to the so-called free boundary;
   i.e., an unknown boundary between continuation and conversion regions of the bond.
    The characteristic feature of such a bond, with an extra call feature,
     is that the free boundary may reach all the way to the fixed boundary. Our intention in this paper is to study the
    behavior of the free boundary in the vicinity of a touching point with the fixed boundary.
     Along the lines of our analysis we will also produce some results on regularity of solutions (value of the bond) up to the fixed boundary.

Our methods are robust and of general nature, and can be applied to fully nonlinear equations.
 In particular, we shall obtain uniform results for the regularity of both solutions and  their free boundaries.
\end{abstract}

\section{Introduction and Backgrounds}

The goal of this paper is to investigate some properties of the solution and the free boundary arising from pricing  convertible bonds with the additional call feature, by PDE methods. But first we overview some basic financial notions, we are dealing with, in this paper.

\subsection{What are Convertible Bonds?}

 \emph{Bond} is a contract which is paid in advance and yields a specified amount on a known date, which is usually called maturity (expiry) date. It is commonly issued by
the  government or major companies that can guarantee the pay back of the predetermined amount to the holder, on the maturity date. Bonds may also pay a known cash dividend, commonly named coupon, at intervals up to and including the maturity time. 
   
Face value for a bond (par value) refers to the amount paid to the holder at the maturity time. In this paper we shall take the face value to be $K$.

 \emph{Convertible Bond} is a bond that entitles the holder the right  to convert it into an agreed-upon amount of the company stock or asset, usually at any
 time of the holder's selection  and sometimes at certain times during its lifespan.

\emph{Call feature} refers to the issuing company's right to buy back the bond for an amount $K$; it is by intention we have taken the
 face value and the call value equal. Clearly, the price of the convertible bond $V$ is
 less than this amount, i.e., $V(x,t) \leq  K$ (otherwise arbitrage to the benefit of the bank will take place). A convertible bond  with call feature is worth less than a standard one. 

\subsection{Pricing Convertible Bonds} \label{PCB}
Let $x=x(t)$ be the price of the stock at time $t$.
 Consider a portfolio, consisting of being long one convertible bond and short a number $\triangle$ of the stock at time $t$,
$$\Pi = V- \triangle x.$$
Suppose that the stock price follows a lognormal random walk, $dx=r x \,dt+ \sigma x \, dW$, where, $\sigma$ is the volatility of the stock, and  $r$ indicates the interest rate and $W$ is a Wiener process. Applying Ito's formula, one can derive the changes in portfolio's value
$$d\Pi=\frac{\partial V}{\partial t}dt+ \frac{\partial V}{\partial x}dx+\frac{1}{2} {\sigma}^2 x^2\frac{\partial ^2 V}{\partial x^2}dt-\triangle qx dt-\triangle dx+cdt,$$
where $c$ and $q$ are the coupon payment of the bond and the dividend of the stock respectively (see \cite{W} for more detailed calculation).
Since the return should be at most that of a bank deposit and according to the fact that $d\Pi=r \Pi dt$, we conclude
$$
 r(V-\triangle x) dt \geq\frac{\partial V}{\partial t}dt+ \frac{\partial V}{\partial x}dx+\frac{1}{2} {\sigma}^2 x^2\frac{\partial ^2 V}{\partial x^2}dt-\triangle qx dt
-\triangle dx+cdt.
$$
In order to have a risk-free portfolio, using $\Delta-$hedging we consider $\triangle=\frac{\partial V}{\partial x}$.
Therefore
$$\frac{\partial V}{\partial t}+\frac{1}{2} \sigma^2 x^2 \frac{\partial ^2 V}{\partial x ^2}+(r-q)x\frac{\partial V}{\partial x}-rV+c(x,t)\leq 0,$$
with the terminal condition
$$V(x,T)=K.$$
On the other hand, at any time prior to expiry, the bond may be exchanged into $\gamma$ number of shares, so
$$V\geq \gamma x,$$
where $\gamma$ is called conversion factor.
Furthermore when the stock price moves up, it is apparent that the bondholder's will is to convert the bond into the pre-determined number of shares rather than getting the face value on the expiry date;
$$ V\rightarrow \gamma x \qquad \hbox{as} \qquad x\rightarrow \infty.$$
Conversely if there is no underlying asset, i.e., $x=0$, then the bond value satisfies the  equation
$$ \partial_t V + c = rV , $$
with termination value $V(0,T)= K$,
and hence admits a solution 
\begin{equation}\label{termination}
V(0,t)=K e^{-r (T-t)}+ \frac{c}{r} ( 1- e^{-r(T-t)}).
\end{equation}
Even though this seems obvious from a financial point of view, we can not claim this from a PDE point of view (at least it is not straightforward). Such a claim involves a verification of 
\begin{equation}\label{bdry-regul-0}
\lim_{x\to 0}x^2V_{xx}(x,t)=\lim_{x\to 0}xV_{x}(x,t)=0 ,
\end{equation}
for a solution to our problem. This in general may fail.
In this paper we shall consider solution only in a class with property (\ref{bdry-regul-0}). This fact (or mentioning of it) is ignored many times in the existing literature.


\subsection{Notation}
We shall use the following notations in this paper.
\begin{equation*}
\begin{array}{rl}
 X: & (x,t),\\
 K: & \text{Call price and Face value of the bond,}\\
    & \text{in this paper both are equal} ,\\
 c: & \text{Coupon rate (related to the bond),}\\
 q: & \text{Dividend rate (related to the stock),}\\
 r: & \text{Interest rate,}\\
 \gamma: & \text{Conversion factor,}\\
 V(x,t): & \text{Price of the Convertible Bond,}\\
 D_T: & (0,\frac{K}{\gamma})\times (0,T) ,\\
 B_\rho: &\{y\in \mathbb{R}, \,\,\mid x-y\mid < \rho \},\\
 Q_\rho(X): & B_\rho \times (t-\rho^2,t+\rho^2),\\
  Q^-_\rho(X) :& Q_\rho(X) \cap {\mathbb{R}}^- ,\\        
 \Lambda: &\left\lbrace (x,t): V=\gamma x \right\rbrace ,\\
 \Gamma: &\partial \left\lbrace V> \gamma x\right\rbrace \cap D_T \,\, (\text{The free boundary}),\\
 \partial_p :& \text{Parabolic boundary},\\
 \text{Continuation region}: &\{ (x,t): \gamma x<V(x,t)< K\},\\
 \text{Call region}: &\{(x,t): V(x,t)= K\},\\
 \text{Conversion region}: &\{(x,t): x\geq 0,\quad V=\gamma x\}.\\
\end{array}
\end{equation*}

\subsection{Problem Statement }
 Let the operator  $ \mathcal{L}$ be as follows
 \begin{equation}\label{l}
 \mathcal{L}=-\partial_t -\frac{1}{2}\sigma^2 x^2 \partial_{xx} -(r-q)x\partial_x +r.
 \end{equation}
We shall consider the convertible bond with the  extra call feature. If $K$ is the call price, set by the firm, then theoretically we have
$$\gamma x \leq V\leq K.$$

Hence we have the following two constraint variational inequality
\begin{eqnarray*}
  \begin{cases}
\mathcal{L}V=c,  &\{\gamma x <V< K\}\cap  D_T, \\
\mathcal{L}V\geq c,  & \{ V=\gamma x\} \cap  D_T, \\
\mathcal{L}V\leq c,  & \{ V=K \} \cap D_T, \\
V(K/\gamma ,t)=K, & 0 \leq t \leq T, \\
V(x,T)=K,        &         0\leq x \leq K/\gamma ,\\
\end{cases}
\end{eqnarray*}
where $D_T=(0,\frac{K}{\gamma})\times (0,T) $. To give  meanings to the above equations we shall look for solutions
 $V$  in the class $ W^{2,2}_{x}\cap W^{1,2}_{t} (D_T)$. 
Observe that taking the equation in this sense forces a smooth fit along the free boundary $\partial \{V> \gamma x\}$.
 From a financial point of view it is natural to assume that the value function $V$ should be continuously differentiable  in order to avoid arbitrage.
  From a PDE point of view this falls out naturally  when  variational formulation is considered
$$
\max(\min(\mathcal{L}V - c, V -\gamma x), V - K)=0,
$$
where the equation is in appropriate sense (weak or viscosity) and with the boundary values as above. It is however not our intention to discuss such aspects of this problem here, but rather pay attention to the qualitative behavior of the free boundary when (and if) it approaches the fixed  boundary $x=K/\gamma$. From the variational inequality approach--or any other for that matter one always obtains the fact that $\mathcal{L}V\geq c$ in $  D_T \cap \{V< K\}$.
 A very good source for  such problems from both Stochastic and Variational point of view is Avner Friedman's book \cite{Fr}, Chapter 16.

It should be remarked that in the above setting we did not assign boundary values to the characteristic boundary $x=0$, due to the 
 degeneracy of the equation at such points. Degenerate  points are treated 
by the problem as interior points and one can not assign boundary data; see \cite{Fich}, \cite{OR}, and \cite{KN}. 
Nevertheless, using the assumption (\ref{bdry-regul}), the value at $x=0$ can be computed directly from the equation, as we did above
(see (\ref{termination}))
$$
V(0,t)=K e^{-r (T-t)}+ \frac{c}{r} ( 1- e^{-r(T-t)}) \qquad   0\leq t \leq T.
$$
One can also treat the problem by considering a regularization of the equation in $\{-K/\gamma <x < K/\gamma ,\  0<t<T\} $ where in the set $\{-K/\gamma <x < 0,\  0<t<T\} $ we consider the evenly reflected coefficients of the operator, which amounts to 
 $$  -\partial_t -\frac{1}{2}\sigma^2 (x^2+ \epsilon) \partial_{xx} -(r-q) x \partial_x +r, $$
where $\epsilon >0$.

In this new setting the problem admits a solution (call it $W_\epsilon$) and by uniqueness and symmetry of the problem, the solution is symmetric in $x$-variable, and hence 
\begin{equation}\label{Wx}
\partial_x W_\epsilon (0,t)=0 .
\end{equation}
Now if we have enough regularity for solutions 
we may  let $\epsilon $ tend to zero to obtain a solution to our original problem. Indeed, such  uniform estimates for $W_\epsilon$ in compact sets of $D_T$ is true by classical theory for uniformly parabolic equations. This convergence is not obvious on the boundary $\{x=0 \} \cap \partial D_T$. In particular we {\it can not} claim that $\partial_x V(0,t)=0$, at least not without any further analysis. 
Property (\ref{Wx}) will be used in proving that the function $V$ is monotone increasing in $x$-variable but  grows slower than $\gamma x$. See Proposition \ref{T1}.

Another approach to existence of this type of  degenerate problem is the use of a mapping that send the origin to $-\infty$, so that the degeneracy appears at infinity point. Then by solving the problem in finite domains, and letting the domain enlarge to the whole space one obtains a solution.

\begin{assu} \label{as} Throughout the paper we shall assume certain conditions to be fulfilled. These conditions force the problem to behave correctly, in a certain sense that is crucial for the problem; e.g. standard maximum principle, uniqueness, compactness should be available. These conditions will in general lead us to existence, uniqueness and qualitative properties which are expected. It should be mentioned that, as we shall discuss it later, these assumptions fall natural for the problem from a financial standpoint. For parabolic PDE  we refer to  \cite{L} as a general background reference.

\begin{itemize}
\item   Call feature: The call feature, as we explained above, determines  the region $0<x<K/\gamma$ where the equations take place.

\item    We shall also assume  $c< rK$. This assumption forces the  value function $V$ to stay strictly below $K$ and hence the upper-obstacle never takes place in $D_T$. This can, indeed, be seen from  equation (\ref{1}). Since in the set $\{V=K\}\cap D_T$ we would then have $\mathcal{L}V=rK \leq c $ contradicting the assumption $c<rK$. It should however be remarked that this assumption does not affect the general results  in this paper since if $c\geq rK$ one
may consider the obstacle problem with upper obstacle only. When these ingredients vary in $(x,t)$ and $c-rK$ changes sign,   then the techniques in this paper has to be modified substantially.

  \item We shall further assume that $c<qK$ otherwise there would be no touch between the graph of $V$ and that of $\gamma x$, i.e., no
  free boundary.  This can be  seen from equation (\ref{lower-bound}) below. Indeed from  (\ref{lower-bound}) we  have $\Lambda \subset \{q\gamma x >c\}$ so for $x$ on  the free boundary we have this condition as well, and therefore $x>c/q\gamma$.
  On the other hand if $c>qK$ then $x>K/\gamma$ and therefore $x$ is  outside the region $D_T$.
   
\end{itemize}
\end{assu}
In most literatures there is an extra  assumption  $q\leq r$, which addresses the fact that if dividend rate are higher than interest rate then no body will invest in the bond market.

Since we are just dealing with the case in which $rK>c$, when $V$ is strictly smaller than $K$,
we may skip the upper obstacle in the formulation of the problem. Also the variational problem produces a supersolution to the PDE and therefore $\mathcal{L}V\geq c,$ holds in $ D_T$.
We thus   look into the following equations of variational inequalities
\begin{equation}
 \label{1}
\begin{cases}
\mathcal{L}V=c,  &D_T \cap  \{ \gamma x <V< K \}, \\
\mathcal{L}V\geq c,  &  D_T, \\
V(K/\gamma ,t)=K, & 0 \leq t \leq T, \\
V(x,T)=K,        &         0\leq x \leq K/\gamma ,\\
\end{cases}
\end{equation}
with a  smooth fit along the free boundary (see Proposition \ref{T1} for more explanation). In other words our solution to the above problem will be $C^{1,\alpha}_x \cap C_t^\alpha$, and the corresponding Sobolev spaces are $W_x^{2,p}\cap W_t^{1,p}$ for $1<p<\infty$. 

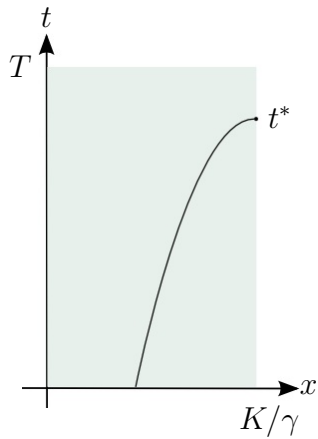
\begin{figure}[!ht]
     \includegraphics[scale=0.5] {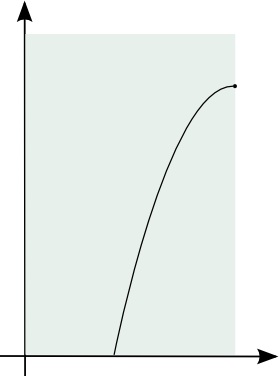}
\put(-110,125){${T}$}
\put(-98,144){$t$}
\put(0,5){$x$}
\put(-12,106){$t^*$}
\put(-23,-8){$K/\gamma$}
    \caption{ The free boundary arising from Problem (1).}
    \label{f1}
\end{figure}

An interesting way of thinking is the fact that Black-Scholes formula determines the value of the bond not only for the current and future times, but it can also evaluate the bond in the past, i.e.,  for $ - \infty <t <T$.

  In the following proposition the existence and the uniqueness of the solution is formulated and then we show the monotonicity and regularity in both $t$ and $x$ directions.

\begin{prop}\label{T1}
There exists a solution $V \in W^{2,p}_{x, loc}\cap W^{1,p}_{t, loc}\,$, $\, 1<p<\infty$ to Problem (\ref{1}).
The solution is unique in the class of all solutions with 
\begin{equation}\label{bdry-regul}
\lim_{x\to 0}x^2V_{xx}(x,t)=\lim_{x\to 0}xV_{x}(x,t)=0 ,
\end{equation}
and satisfies
\begin{equation}\label{monot-V}
0\leq   V_t , \qquad  0\leq  V_x\leq \gamma.
\end{equation}
Consequently, in this class we also have that the exercise region (if non-empty) is an epi-graph in both $x$ and $t$ directions. 
\end{prop}
Although the statements of this theorem is well-known to experts, and is considered as "common property", we shall sketch a proof of  it in Section \ref{pr}. It is worth mentioning that in most known literatures the assumption (\ref{bdry-regul}) is not taken into account and this may 
cause serious problems. This assumption, however, can be relaxed to boundedness and then uniqueness will be in the class of variational solutions.

An important feature for convertible bonds is the presence of the free boundary, and the so-called exercise region $\Lambda$. If one chooses the ingredients in the problem in a way that conversion is never  optimal  
 then this falls under standard theory of bonds. We shall avoid such a discussion here, but  we assume from now on that
\begin{equation}\label{Lambda}
  \Lambda\neq \emptyset .
\end{equation}

Once we agree upon condition  (\ref{Lambda}) then we need to make sure that the touching point between the free boundary and the fixed boundary $t=t^*$ exists as well as it 
does not come too close to the termination point $t=T$.

\begin{prop}\label{pro} In Problem (\ref{1}), there is $t^* \in (-\infty, T)$ such that the free boundary $\Gamma=\partial\{V>\gamma x\}$, hits the fixed boundary only at  $(K/\gamma,t^*)$
and at no other fixed boundary point.
\end{prop}

To not digress from the discussion, we shall prove this proposition later in Section (\ref{pr}).

\begin{rem}
In Problem (\ref{1}), $V_x(K/\gamma,T)=0$ and $V_x(K/\gamma,t^*)=\gamma$. Moreover $V_x \in C_t ^\alpha(\overline{D}_T \setminus \{x=0\})$ (see Theorem \ref{T3}), hence for every $x$ there is a $c_0$, such that
$$\mid V_x(x,t^*)-V_x(x, T) \mid \leq c_0 (T-t^*)^\alpha.$$
From here we deduce that there is a certain distance between $t^*$ (in Proposition \ref{pro}) and the maturity date, more exactly
\begin{equation}\label{t*}
T-t^*\geq (\frac{\gamma}{c_0})^{\frac{1}{\alpha}}.  
\end{equation}
 \end{rem}

\section{Main Results}\label{main}

In this section we shall state our main results concerning qualitative behavior of the solution to Problem (\ref{1}) and the free boundary arising from it. The proofs are also gathered in the next section.

The following theorem states the optimal smoothness of the solution to Problem (\ref{1}).

\begin{theo} \label{T3}The solution $V$ to Problem (\ref{1}) is uniformly $C_x ^{1,1} \cap C_t^{0,1}$ in  $\overline{D}_T \setminus \{x=0\}$.
\end{theo}

Finally we state a theorem, which shows that the free boundary is situated outside the set $\left\{(x,t), t> -\alpha |x-K/\gamma|^2+t^* \right\} \bigcap D_T$. In other words, it is located under any arbitrary downward parabola at the point $(K/\gamma,t^* )$  and after scaling it tends to the line $x=K/\gamma$.

 \begin{theo}\label{T4}
 The free boundary $\Gamma=\partial \{V>\gamma x\}$ in Problem (\ref{1}), is  uniformly parabolically tangential to the fixed boundary at
  $X^*=(K/\gamma, t^*)$. In other words there is a modulus of continuity $\sigma$ ($\sigma(0^+)=0$) and  an $r_0$  such that
$$\Gamma \cap Q_{r_0}(X^*) \subset \left\{(x,t) : t-t^*\leq -\frac{|x-K/\gamma|^2}{\sigma(|x-K/\gamma|)} \right\}.$$
\end{theo}

The reader should notice that Theorem \ref{T4} implies that a parabolic scaling $(sx ,s^2t)+ X^* $ of the free boundary at $X^*$ gives that the limiting free boundary 
$\Gamma_0=\lim_s \Gamma_s$ (where $\Gamma_s$ is the scaled free boundary) will coincide with the $t$-axis.

 \begin{rem} 

  \begin{itemize}
\item   It is notable that the free boundary arising from the American option behaves in a completely different way, i.e., it is located above all arbitrary parabolas and after scaling it tends to the spatial axis.

\item A final remark concerning Theorem \ref{T4} is that here we have  intentionally avoided to discuss the regularity of the free boundary due to  technical reasons. Indeed, one expects the free boundary in this case to be smooth (up to $C^\infty$) but a proof of this "fact" needs detailed analysis and blow-up techniques. There are several papers treating regularity of the free boundary 
  for the american put/call option and it would be likely that these methods apply here, even though not directly.
  Such a sketch of ideas is presented in \cite{PS}.

\item We want to further remark that   all estimates and statements in our main results, hold in a uniform fashion. More exactly the constants 
in our theorems above depend only on the norms of the ingredients and not the ingredients themselves. Nevertheless we shall only do the proof for one solutions and not a general class of solutions, as it would require more definitions and complications.

\end{itemize}
\end{rem}

\section{Restatement of the problem and technical tools}
In this section we shall make change of variables so that the main equation falls under general theory of parabolic PDE.
We shall also state some general facts from the standard theory of the free boundary regularity, for parabolic equations (see \cite{PS}, \cite{ASU1}, and \cite{ASU2}).\\
\begin{dis}\label{dis}
For the sake of simplicity first we translate $V$ to $\tilde{V}$ by
$$\tilde{V}(x,t)=V(-x+K/\gamma,T-t).$$
Inserting this in equation (\ref{1}), gives
\begin{small}
\begin{equation}
 \label{2}
\begin{cases}
\mathcal{\tilde{L}}\tilde{V}=c,  &D_T \cap  \{ -\gamma x+K <\tilde{V}  \} , \\
\mathcal{\tilde{L}}\tilde{V}\geq c,  &  D_T, \\
\tilde{V}(0 ,t)=K, & 0 \leq t \leq T, \\
\tilde{V}(x,0)=K,        &         0\leq x \leq K/\gamma ,\\
\end{cases}
\end{equation}
\end{small}
where $\mathcal{\tilde{L}}$ is the following operator
\begin{equation}\label{L-operator}
\mathcal{\tilde{L}}=\partial_t-\frac{1}{2}\sigma^2(-x+K/\gamma)^2\partial_{xx}+(r-q)(-x+K/\gamma)\partial_{x}+r .
\end{equation}
One may also derive from (\ref{termination})
\begin{equation}\label{x=0}
\tilde{V}(K/\gamma ,t)=K e^{-r t}+ \frac{c}{r} ( 1- e^{-rt}),   \qquad         0\leq t \leq T.
\end{equation}
The obstacle should also be transformed and it becomes $-\gamma x+K$. Now for
$$u(x,t)=\tilde{V}(x,t)+\gamma x-K,$$
we have
$$\mathcal{\tilde{L}}u=c+(r-q)(-x+K/\gamma)\gamma+r(\gamma x-K)=c-q(-\gamma x+K),$$
 in the set $\{0<u<\gamma x\}.$
Moreover $u$ satisfies the following problem
\begin{equation}
\label{3}
\begin{cases}
\mathcal{\tilde{L}}u=c-q(-\gamma x+K),  & D_T \cap \{ 0<u  \} , \\
\mathcal{\tilde{L}}u\geq c-q(-\gamma x+K),  &  D_T, \\
u(0 ,t)=0, & 0 \leq t \leq T, \\
u(x,0)=\gamma x,        &         0\leq x \leq K/\gamma .\\
\end{cases}
\end{equation}
One also observes that from (\ref{x=0}) we have
$$
u(K/\gamma , t)= K e^{-r t}+ \frac{c}{r} ( 1- e^{-rt}),  \qquad         0\leq t \leq T .
$$

We also define a  scaled operator
\begin{equation}\label{scale-op}
\mathcal{\tilde{L}}_{s,X^0}:=\partial_t-\frac{1}{2}\sigma^2(-(sx+x^0)+K/\gamma)^2\partial_{xx}+s(r-q)(-(sx+x^0)+K/\gamma)\partial_{x}+s^2r,
\end{equation}
and its corresponding  scaled function
$$v(x,t)=\frac{u(sx+x^0,s^2t+t^0)}{A_s},$$
at the point $X^0=(x^0,t^0)$, for some fixed $s$, also in the forthcoming proofs, $A_s$ will either be $\sup_{Q_s} u$ or $s^2$, dependent on the situation.
One can easily verify that
\begin{equation}\label{scale-eq}
\mathcal{\tilde{L}}_{s,X^0} v(x,t)=\frac{s^2}{A_s}\,(\mathcal{\tilde{L}}u) (sx+x^0,s^2t+t^0).
\end{equation}
\end{dis}
\begin{figure}[!ht]
     \includegraphics[scale=0.5] {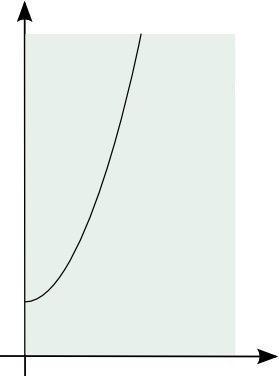}
\put(-127,25){$T-t^*$}
    \caption{ The free boundary after transformation}
    \label{f1}
\end{figure}

\begin{lem}\label{lem1} The solution $u$ to Problem (\ref{3}) satisfies
\begin{equation}\label{transf-eq1}
\mathcal{\tilde{L}}(u)=(c- q(-\gamma x + K))\chi_{\{u>0\}} \qquad \hbox{in } D_T,
\end{equation}
along with the boundary conditions
\begin{equation}\label{transf-eq2}
u(0,t)=0 \quad \hbox{for } 0\leq t \leq T,  \qquad u(x,0)=x\gamma  \quad \hbox{for } 0\leq x \leq K/\gamma.
\end{equation}
\end{lem}

\proof 
The proof follows now easily due to smoothness of the function $u \in W_x^{2,p}\cap W_t^{1,p}$ locally in $D_T$, and the equations in (\ref{3}). \qed

\begin{lem} \label{L1} (Non-degeneracy for parabolic equations) \label{nd} Let $u$ be a solution of Problem (\ref{transf-eq1})-(\ref{transf-eq2}), then
$$\sup_{Q^-_{\rho}(X^0)} u \geq  c_0 \rho^2 +u(X^0), \,\,\, \hbox{for }  0<\rho < K/\gamma,$$
for all $X^0 \in \overline{\{u>0\}}$.
\end{lem}

\proof

The proof of  non-degeneracy is quite standard and it follows by simple maximum principle (see \cite{CPS}). In this particular case with $u_t = \tilde V_t = -V_t \leq 0 $ we can not  apply the
proof for elliptic case (as is standard for American put/call options); cf. \cite{LS} Lemma 2.2.

 We shall present another proof by defining $h(x,t)=a|x-x^0|^2 - b(t-t^0)$ in the cylinder $Q^-_{\rho}(X^0)$. Let further $a,b$ be chosen appropriately  to have $\mathcal{\tilde{L}} h \leq  c-qK$, so that 
$h$ is a sub-solution to the variational problem locally in $Q^-_{\rho}(X^0)$. Now if $h \geq u$ on the parabolic boundary $\partial_p Q^-_{\rho}(X^0)$ then $h \geq u$ on $Q^-_{\rho}(X^0)$. 
In particular if $X^0 \in \{u>0\}$ then we arrive at the contradiction
$0=h(X^0)\geq  u (X^0)>0$. 
By continuity we may let $X^0$ tend to a free boundary point so that the lemma holds for all points in the set $\overline{\{u>0\}} \cap D_T$. \qed

The idea of the proof of this lemma appeared in a paper by G.S. Weiss (see Appendix in \cite{We}) for the elliptic case.

\section{Proofs of the main results} \label{pr}

\subsection{Proof of Proposition\ref{T1}} 
We split the proof in cases, according to statement. The proofs being standard, we shall only sketch the ideas. We shall also consider the problem for the  translated function $\tilde V$.

\vspace{3mm}
\noindent
{\it {Existence}}:\\
To prove the existence of the solution, one can utilize the regularization/penalization or Perron's  method, 
this is standard and follows from various literature, e.g. \cite{F}, also see \cite{YY}. Observe that due to degeneracy of the PDE one needs more care in this part, that the usual uniformly elliptic/parabolic case. Therefore an approximation of the problem to a uniformly parabolic problem or alternatively a domain change by a transformation   will be needed.

\vspace{3mm}
\noindent
{\it Partial Regularity}:
This  is also quite classical and one can easily prove (local) $C_x^{1,\alpha} \cap C_t^{0,\alpha}$ regularity in the literatures. It should be noticed that once the reformulation of the problem as in Lemma \ref{lem1}  is done  then one obtains standard 
parabolic equation with bounded right-hand side. Classical PDE then tells us that the solution is $C_x^{1,\alpha} \cap C_t^{0,\alpha}$, for any $0<\alpha <1$. Observe that without the assumption (\ref{bdry-regul}) one can not obtain such regularities up to the boundary $\{x=0\}$ using classical techniques.

\vspace{3mm}
\noindent
{\it Uniqueness}:\\
To show the uniqueness of the solution, we suppose that there exist two solutions $\tilde{V}_1$ and $\tilde{V}_2$ to Problem (\ref{2})
and define the set $\Omega$ as follows
$$\Omega:=\{X=(x,t); \,\, \tilde{V}_1(X)>\tilde{V}_2(X)\}\subset D_T.$$
We have $\tilde{V}_1> \tilde{V}_2\geq -\gamma x+K$   in $\Omega$, hence $\mathcal{\tilde{L}}(\tilde{V}_1)=c$ in $\Omega$. Moreover $\mathcal{\tilde{L}}(\tilde{V}_2)\geq c$, i.e., $\tilde V_2$ is a super-solution to the equation and therefore we can use the comparison principle (see \cite{F} and \cite{L}) to conclude that $\tilde{V}_2\geq \tilde{V}_1$ in $\Omega$,  which leads us to a contradiction. Observe that in the above argument we may use the  values at  $x=K/\gamma$, as both solutions should have the same (non-assigned) value at $x=K/\gamma$; see (\ref{x=0}).

\vspace{3mm}
\noindent
{\it Monotonicity}:\\
In order to prove the monotonicity in $t$-direction we slide the solution $\tilde{V}$ slightly in $t$-direction and obtain $\tilde{V}_\alpha$
$$\tilde{V}_\alpha(x,t)=\tilde{V}(x,t+\alpha).$$
We want  to show $\tilde{V}_\alpha\leq \tilde{V}$, in $D_{T-\alpha}$ for  $\alpha >0$.

We have (by inspections)  $\tilde{V}_\alpha \leq \tilde{V}$ on the boundary of   $D_{T-\alpha}$ (where we may also use the value (\ref{x=0})) and therefore according to comparison principle (see the uniqueness part) we can deduce that $\tilde{V}_\alpha \leq \tilde{V}$ in the  domain $D_{T-\alpha}$. Observe that the operator does not change, while shifting in $t$-direction as the coefficients are independent of $t$.

For  the monotonicity in $x$-direction, a similar  slide in $x$-direction does not work, since the operator will change. We may does apply the maximum/minimum principle to $\partial_x W_\epsilon$ (introduced earlier, see (\ref{Wx})) in the continuation (non-coincidence set) for $W_\epsilon$. 

We first consider $\tilde W_\epsilon = W_\epsilon (-x+K/\gamma, T-t)$ from (\ref{Wx}) in the region $D_T$, and observe that $\partial_x \tilde W_\epsilon (K/\gamma, t)= \partial_x W_\epsilon (0,t)=0$.

Next we  apply the maximum principle to the equation 
$\partial_x \mathcal{\tilde{L}}(\tilde W_\epsilon) = \partial_x c$ in the region $\{\tilde W_\epsilon  > -x\gamma + K\}$, and observe that $\tilde V_x$ satisfies 
$$
\partial_t(\partial_x \tilde W_\epsilon)-\frac{1}{2}\sigma^2(-x+K/\gamma)^2\partial_{xx} (\partial_x \tilde W_\epsilon) +(\sigma^2+ r-q)(-x+K/\gamma) \partial_{x} (\partial_x \tilde W_\epsilon)  +q (\partial_x \tilde W_\epsilon) = 0 ,
$$ 
in the region $\{ \tilde W_\epsilon > -x\gamma + K\}$; compare (\ref{L-operator}).

On $\{t=K/\gamma\}$ we have by (\ref{Wx}) that $ \partial_x \tilde W_\epsilon (K/\gamma,t)= - \partial_x  W_\epsilon(0,t)=0$, and for 
$ \partial  \{\tilde W_\epsilon > -x\gamma + K\} $ as well as for $\{t=0\}$ 
we have $\partial_x  \tilde W_\epsilon = -\gamma$. Finally on $x=0$ with $\{0<t<T-t^*\}$ we have 
$-\gamma \partial_x  \tilde W_\epsilon \leq 0$. The last inequality depends on the fact that $\tilde W_\epsilon \geq  -x\gamma + K$.

The maximum and minimum principle both apply and we obtain $-\gamma \leq \partial_x  \tilde W_\epsilon  \leq 0$. This naturally shows that the coincidence set (exercise region) for the $\epsilon$-problem is an epi-graph. As $\epsilon$
tends to zero, $W_\epsilon$ tends to $V$ and we obtain the same properties for $V$ and its graph.

\subsection{Proof of Proposition\ref{pro}}

    We recall from (\ref{Lambda}) that the  $\Lambda \neq \emptyset$.  By monotonicity of  $V$  in both $x$ and $t$-directions (Proposition \ref{T1})
    we conclude that  the set $\Lambda$ is connected, and its boundary is an epi-graph, in both spatial and time directions.
     In particular $\Lambda= \overline{int(\Lambda)}$ (closure of the interior).
     Hence $\partial \Lambda = \partial (int(\Lambda))$.

Since $V=\gamma x$ in the interior of $\Lambda$ we may differentiate to obtain
$$\partial_t V=\partial_{xx} V=0 \quad \hbox{and} \quad \partial_x (\gamma x)=\gamma .$$
Implementing these in equation $c\leq \mathcal{L}V$ gives $c\leq -(r-q)x\gamma + r \gamma x $.
Therefore $c\leq q\gamma x$ for all points in the interior of $\Lambda$. Hence there exists a lower bound for $x \in \Lambda $, i.e., 
\begin{equation}\label{lower-bound}
\Lambda \subset \left\{x: x\geq  \frac{c}{q \gamma}\right\}.
\end{equation}
This shows that the free boundary can not touch the time axis.

Next, if $\Lambda$ reaches all the way to $t=T$, with $x< K/\gamma$, then $V(x,T)=K$ and in $\Lambda $ we have $V(x,T) = \gamma x < K$, the continuity of $V$ breaks down, contradicting Proposition \ref{T1}.
From both arguments above we conclude that the free boundary neither touches the $t$-axis, nor any point on $t=T$, with $x<K/\gamma$.
We need now to exclude the point $(K/\gamma, T)$. This is a little bit more tricky and one should use a higher regularity of $V$ up to the corner points.

It is in general well-known that variational problems produce the same amount of regularity for the solutions as that of the given obstacle
in $x$-variable and half regularity in $t$-variable (but  $C_x^{1,1}\cap C_t^{0,1}$ is in general a regularity threshold). 
In particular $\partial_x V$ exists and is continuous  up to the boundary, including the point $(K/\gamma, T)$.
This implies in turn that on one side $\partial_x V(K/\gamma, T)=0$ (derivative along the segment $t=T$) and on the other side $\partial_x V(K/\gamma, T)=\gamma$
 (derivative along the segment $x=K/\gamma$). This is a contradiction, and hence there is a $t^*< T$ such that $\Lambda \subset \{t<t^*\}$.

\subsection{Proof of Theorem \ref{T3}.}

We first notice that due to (\ref{t*}), (\ref{lower-bound}) and the monotonicity of the $V$ (Proposition \ref{T1}), the function $V$ solves an standard parabolic equation in the set
$$ \{ x < c/q\gamma\} \cup \{t> T - (\gamma /c_0)^{1/\alpha }\}.$$
Hence  we can conclude by classical parabolic theory that the $V$ is as smooth as stated in the theorem, up to the fixed boundary (the ingredients are smooth enough).

Next we will prove  the following  lemma, which says that if the solutions grows quadratically from the free boundary then one can after rescaling obtain uniform regularity
as stated in Theorem \ref{T3}. We state the lemma in terms of the function $u$, i.e., a solution to equation (\ref{3}), or more exactly a solution to (\ref{transf-eq1})- (\ref{transf-eq2}).

Let us now introduce some notations:
$\overline{\Gamma}:=\Gamma \cup t^* , $
$$
d^- (X, \partial D_T)= \sup\{\rho : Q_\rho^- (X)\subset D_T \}, 
$$
$$
d^-(X,\Gamma)= \sup\{\rho: Q_\rho^- (X)\subset D_T\setminus \Lambda \}.
$$

\begin{lem} \label{lem2} Under the assumptions of Theorem \ref{T3},  let $u$ be a solution of equation (\ref{transf-eq1})- (\ref{transf-eq2}).
Suppose also  that for $\rho >0$
\begin{equation}\label{growth-u}
\sup_{Q'_{\rho}(X^0)} u \leq C_0 \rho ^2, \qquad (Q'_{\rho}=Q_{\rho} \cap D_T)
\end{equation}
for all $ X^0 \in \overline{\Gamma} $, and a universal $C_0$ independent of the ingredients in the equations.
Then $$u \in \left( C^{1,1}_x \cap C^{0,1}_t\right) (\overline{D_T} \setminus \{x=K/\gamma \}) ,$$
where the norm depends only on the norms of the ingredients.
\end{lem}
\begin{proof}
Consider a point $X=(x,t)\in D_T \setminus \Lambda$, set   $s:=d^-(X,\Gamma)$,  and let $Y=(y,\theta) \in Q_1(0)$.
According to the discussion preceding the lemma, $u$ is regular in $D_{T-t^*}$. Hence   we may assume $t\geq T-t^*$ (see figure 2). 

Define $$W(Y)=\frac{u(x+sy,t+s^2 \theta)}{s^2}, \qquad \text{in}\,\,\, Q_1(0).$$
By (\ref{growth-u}) we have
$$\| W\|_{L^\infty (Q_1(0))}\leq \sup_{Q_s(X)} \frac{u}{s^2} \leq \sup_{Q_{2s}(\tilde X)} \frac{u}{s^2}
\leq 4C_0. $$
Here $\tilde X \in \Gamma$ is the closest point to $X$ on the free boundary.
 
Now $W$ being bounded in $Q_1(0)$ and $\mathcal{\tilde{L}}_{s,X}W=c ,$
we can apply interior Schauder estimate, to arrive at 
$$|\partial_{xx} W(0,0) |+| \partial_t W|(0,0) \leq C.$$
On the other hand  $ \partial_{xx} W(0,0)= \partial_{xx} u(X)$ and $ \partial_t W(0,0)= \partial_t u(X)$, which implies
$$
u \in \left(C^{1,1}_x \cap C^{0,1}_t\right) (\overline{D_T}) ,
$$
as stated in the lemma. 
\end{proof}

To prove Theorem \ref{T3} we shall only need to show the growth estimate (\ref{growth-u}) for small $\rho$, as the estimate holds for 
$ d^-(X,\partial D_T)/4 < \rho <d^-(X,\partial D_T) $. The proof of estimate (\ref{growth-u}) will follow using 
 the well-known scaling method and blow-up technique, standard in recent theory of  free boundary regularity (cf. \cite{ASU1}, \cite{ASU2}, and \cite{CPS}).

Set 
 $$S_j(X^0,u)=\sup_{Q'_{2^{-j}(X^0)} } u, \qquad Q'_{2^{-j}(X^0)} = Q_{2^{-j}(X^0)} \cap D_T .$$
We claim that for all $j \in \N$, and all $X^0 \in \Gamma$, and any solution $u$ to our equation 
\begin{equation}\label{S_j}
S_{j+1}(X^0,u)\leq \max \left\lbrace  4^{-j} C_1, 4^{-1} S_j(X^0,u),...,4^{-j}S_1(X^0,u) \right\rbrace ,
\end{equation}
for a universal constant $C_1$, depending only on the norms of the ingredients.

Suppose this is true, then we see that (\ref{growth-u}) follows by inspection. Indeed, 
for any $\rho $ (small) we may choose $j$ such that $ 2^{-j-1}\leq \rho < 2^{-j}$. From (\ref{S_j}) then it follows that
$ S_\rho \leq S_{j} $. Now if the maximum of the right hand side in (\ref{S_j}) happens for $4^{-j} C_1$ then we are done with $C_0=4C_1$. If not then
the maximum is $4^{-k-1} S_{j-k}$ for some $k$, which implies
$ S_\rho \leq S_{j} \leq 4^{-k-1} S_{j-k}$ and we can repeat the argument for $S_{j-k}$, until $k=j$.

We shall now prove (\ref{S_j}), using a contradictory argument.
Hence suppose (\ref{S_j}) fails. Then, for every $n \in \N$, there exist $X^n \in \overline \Gamma$,   and $j_n \in \N$ such that
\begin{equation}\label{contra}
S_{j_n+1}(X^n)>\max \{  n4^{-j_n},  4^{-1} S_{j_n}(X^n),....,4^{-j_n} S_1(X^n)  \},
\end{equation}
for $u$ solving our problem.
Since $S_{j_n+1}(X^n)>  n4^{-j_n}, $ i.e.,  $4^{j_n}S_{j_n+1}(X^n)>  n $ we deduce that
$$j_n \rightarrow \infty \qquad \text{as} \qquad n \rightarrow \infty.$$

 Now set 
 $$  B_n= 2^{j_n}(D_T \setminus X^n)= \{ X: 2^{j_n}(X-X^n) \in D_T\},  $$
 and
 $$v_n(x,t):=\frac{u({2}^ {-j_n} x+x^n,{2}^ {-2j_n} t+t^n )}{S_{j_n+1}(X^n)}, \qquad \hbox{in} \qquad B_n,
$$
 and recall Discussion \ref{dis}.
From (\ref{scale-op}), (\ref{scale-eq}), (\ref{transf-eq1}), and  (\ref{transf-eq2}) we obtain
the rescaled equation 
\begin{equation}\label{eq-v}
\mathcal{\tilde{L}}_n v_n(x,t)=\frac{{2}^ {-2j_n}}{S_{j_n+1}} (\mathcal{\tilde{L}}u) ({2}^ {-j_n}x+x^n,{2}^ {-2j_n}t+t^n)
\qquad \to \quad 0,
\end{equation}
in $B_n.$ Here  $\mathcal{\tilde{L}}_n:= \mathcal{\tilde{L}}_{({2}^ {-j_n}),X^n} $ (see  (\ref{scale-op})), and the right hand side tends to zero follows from (\ref{transf-eq1}), and  (\ref{contra}).

Next for any $R >1$, there is a $m$ such that $2^{m-1}\leq R < 2^m$ and  by using equation (\ref{S_j}) we have
 \begin{equation}\label{vn-uniform}
   \begin{array}{ll}
 \|v_n\|_{L^\infty(Q_R)}&=\frac{\sup _{Q_R} |u(2^{-j_n}x+x^n,2^{-2j_n}t+t^n)|}{S_{j_n+1}(X^n)} \\
 &\leq  \frac{\sup _{Q_{2^m}} |u(2^{-j_n}x+x^n,2^{-2j_n}t+t^n)|}{S_{j_n+1}(X^n,u)}\\
& \leq \frac{\sup _{Q_{2^{m+j_n}}}u}{S_{j_n+1}}\leq \frac{S_{m+j_n}}{S_{j_n+1}}
 \leq 4^{m-1} C_0  \leq C_0 R^2.
 \end{array}
 \end{equation}
 Therefore $v_n$'s are uniformly bounded in any compact subset of $B_n$.
Since the operators $\tilde{\mathcal{L}}_n $ are also uniformly elliptic (in $ B_n$ but 
  away from  $ \{x= 2^{j_n}K/\gamma\}$)  and $v_n$ satisfying the equation (\ref{eq-v}) and they are uniformly bounded by (\ref{vn-uniform})
   we  conclude, by classical Schauder theory, that they are uniformly $C_x^{1,\alpha}\cap C_t^{\alpha}$ in any compact subsets of $B_n$.

In particular by compactness
there is a subsequence (relabeled with the same sequence) such that 
\begin{equation}\label{eq-v2}
v_n \to v_0 \quad  \hbox{ in } B_\infty
\qquad
\qquad
v_0 \geq 0  \quad \hbox{ in } B_\infty
\end{equation}
where $B_\infty = \lim_n B_n$, and $v(0)=0$.

 On the other hand one can easily deduce from the expression for $\mathcal{\tilde{L}}_n$ that
$$\tilde{\mathcal{L}}_n \rightarrow \mathcal{L}_0, \qquad \hbox{as} \qquad n \rightarrow \infty,$$
where $\mathcal{L}_0$ is a scaled version of the heat equation, i.e.,
$$\mathcal{L}_0= \partial _t -\frac{\sigma^2}{2}(-x^0+K/\gamma)\partial_ {xx},$$
and
$$ \mathcal{L}_0 v_0 = 0,$$
with  $v_0(0,0)=0$ and $v_0\geq 0$. It is also crucial to see that since $v_n(0,0)=0$ and $v_n \geq 0$ we have  $\partial_x v_n (0,0)=0$ ($v_n$ is a $C_x^{1}$).
By the fact that the compactness is in the space $C_x^{1,\alpha}$ for a $\alpha >0$ and this is true up to the boundary of the domain (due to smoothness of the boundary) we conclude that 
\begin{equation}\label{hopf}
\partial_xv_0(0,0)=0
\end{equation}
 
Next we see that the domain $B_\infty$ has three possibilities, depending on the value
\begin{equation}\label{limit-value}
\lim_j \frac{d^-(X^{j_n},\partial D_T)}{2^{-{j_n}}} \in \{0,a_0,\infty \},
\end{equation}
for some finite number $a_0>0$.
Now in the case $a_0, \infty$ above we have that the origin is an interior point of the set $B_\infty$ and $v_0$ takes a local minimum, which violates
the minimum principle for caloric functions (or parabolic PDEs). When the limit in (\ref{limit-value})  is zero then $(0,0) \in \partial B_\infty$ and here one can apply the Hopf boundary point lemma (see \cite{L}, Lemma 2.6) to obtain a contradiction to (\ref{hopf}).

\vspace{3mm}
\subsection{Proof of Theorem \ref{T4}.}


 For every given $\epsilon$ define the parabolic-cone $P_\epsilon$ as follows
$$P_\epsilon=\{X=(x,t)\, ;\,\, t> \frac{1}{\epsilon}   x^2 + (T-t^*)  \}.$$
To prove that the free boundary touches the fixed boundary tangentially, we show that for every small $\epsilon$, $P_\epsilon$
contains the free boundary   $\Gamma$ asymptotically close to the point $(0,T-t^*)$.
 We state and prove this in the following lemma, which leads us to the proof of Theorem \ref{T4}.\\

\begin{lem} \label{cone} Let $u$ be a solution to our problem. Then for every $\epsilon > 0$, there exists $r_\epsilon$ such that,
$$\Gamma (u)  \cap Q_{r_\epsilon}(0,T-t^* )\subset P_\epsilon \cap Q_{r_\epsilon}(0, T-t^*). $$
\end{lem}
\begin{proof}
Suppose, towards a contradiction, that the statement in the lemma   fails. Then for every $j \in \N$, there exist
$$X^j=(x^j,t^j) \in \Gamma(u) \cap Q_{r_\epsilon}(0, T-t^*)     \setminus P_\epsilon  \cap Q_{r_\epsilon}(0, T-t^*) ,$$
with $x^j \to 0$, and $t^j \to T-t^*$. Let $s_j=|X^j - (0,T-t^*) |$ (the parabolic distance), and 
define $u_j$ be the scaled version of $u$ at $(0, T-t^*)$ which is defined in the set
$$B_j:=\{(x,t): \ 0<x< K/s_j\gamma , \ -(T-t^*)/s_j < t < t^*/s_j \}.$$
More exactly
$$u_j (X)=\frac{u(s_j x,{s_j}^2 t+T-t^*)}{{s_j}^2}.$$
Now for every scaled function $\tilde{u}_j$, one can find a point $\tilde{X}^j=(x^j/s_j,t^j/{s_j}^2) \in \Gamma(u_j)$ with $|\tilde X^j|=1$.
Next we use Theorem \ref{T3} to see that $ u_j \in C_x^{1,1}\cap C_t^{0,1}$ locally on each compact set of $B_j$. Therefore for some subsequence
$$u_j\rightarrow u_0 \qquad  \text{and}\qquad  \tilde{X}^j\rightarrow X^0$$
were $u_0$ is a global solution in $B_\infty=\{x>0\}$, $X^0 \in \Gamma (u_0) $ with  $|X^0|=1$.

The scaled operator can also be written as $ \tilde{\mathcal{L}}_{j} = \tilde{\mathcal{L}}_{s_j,(0,t^*)}, $ which in turn gives the equation
$$
\tilde{\mathcal{L}}_{j}  u_j = (c -q(-\gamma s_jx + K))\chi_{\{ u_j >0\}}.
$$
As $j \to \infty $ we arrive at a global solution
$$
\tilde{\mathcal{L}}_{0}  u_j = (c -q K)\chi_{ \{u_0 >0\}},
$$
with 
$$\tilde{\mathcal{L}}_{0}= \partial_t - \frac{\sigma^2K}{2\gamma} \partial_{xx},$$
and   $0\leq u_0 (X) \leq  C |X|^2$, $X^0 \in \Gamma (u_0)$. 
Rewriting this we have
$$ \frac{\sigma^2K}{2\gamma}\partial_{xx} u_0  - \partial_t u_0 =  (q K-c)\chi_{\{ u_0 >0\}},$$
which requires $qK-c>0$ for presence of a free boundary; otherwise non-degeneracy can not be applied. 

Here one may scale the operator to reduce it to the heat equation, and then 
use Theorem II in \cite{ASU2} to claim that $\partial_t u_0=0$ (i.e., $u_0$ is time independent). But then $u_0$ is a one dimensional solution to 
$D_{xx}u_0 = A\chi_{\{u_0>0\}}$ for some $A>0$, and that (by non-degeneracy, Lemma \ref{nd}) both origin and $x^0$ (in $X^0=(x^0,t^0)$) are free boundary points.
This is a contradiction as simple computations show that $u_0(x)=Ax_+^2/2$ if the origin is a free boundary point.
\end{proof}

The parabolic tangentiality  of Theorem (\ref{T4}) follows from Lemma (\ref{cone}), by taking the inverse of the relation $\epsilon\rightarrow r_\epsilon$, and denoting it   $\sigma(r)$.

\section{Discussion}
In closing we want to bring to the readers attention several facts and clarifications.

We have assumed that $K,r, c$ and many other ingredients in this paper are constants, which actually is not necessary for our main theorems 
about optimal regularity or the parabolically tangential behavior of the free boundary to work out. All one needs is uniform behavior (in their norm) to
have similar results.

The standing conditions $rK >c$ or/and $qK>c$ may also be dropped but then a double obstacle may occur and if these value also change in time and stock-value $x$ then one may have too complicated situations as the bond value $V$ may touch both upper and lower obstacle and switch between them too many times.
Such an analysis requires deeper insight into the problem. It is noteworthy that double obstacle problems are not so well studied close to a free boundary when both obstacles hit. In our case this is the point $(K/\gamma, t^*)$.  

Other aspects that may be subject for study, by our methods can be the case of convertible bonds with random interest and call feature. The problem now will become two space-dimensional and more delicate. It is however likely (if not apparent) that our methods work in such cases as well and give similar results.

A final remark that we would want to make is the case when the underlying asset is a combination of several assets, and max/min value for such stocks   may be considered as the conversion possibility. 

We hope to come back to such problems for detailed analysis in the future.

\end{document}